\documentclass[12pt]{amsart}
\usepackage{graphicx}
\usepackage{amsmath}
\usepackage{subfigure}
\usepackage{amssymb} 
\usepackage[usenames,dvipsnames]{color}
\usepackage{pinlabel}
\usepackage{mathtools}
\usepackage{float}
\usepackage{tikz}
  \usetikzlibrary{hobby}
  \usetikzlibrary{patterns}

\newtheorem{thm}{Theorem}[section]

\newtheorem{lemma}[thm]{Lemma}

\newtheorem*{defin*}{Definition}

\newcommand{\bdd}{\mbox{$\partial$}}

\begin{document}  

\title{One Powell generator is redundant}   


\author{Martin Scharlemann}
\address{\hskip-\parindent
        Martin Scharlemann\\
        Mathematics Department\\
        University of California\\
        Santa Barbara, CA 93106-3080 USA}
\email{mgscharl@math.ucsb.edu}

\thanks{I am grateful to Michael Freedman for many inspiring and helpful conversations on various aspects of the Powell Conjecture.}

\date{\today}

\begin{abstract}  In 1980 J. Powell \cite{Po} proposed that five specific elements sufficed to generate the Goeritz group of any Heegaard splitting of $S^3$.  This conjecture remains unresolved for genus $g \geq 4$.  Here a short argument shows that one of his proposed generators is redundant, in fact a consequence of three of the other four.
\end{abstract}

\maketitle

Powell proposed generators for the genus $g$ Goeritz group, $g \geq 3$ as shown in Figure \ref{fig:PowellPic}.

\begin{figure}[ht!]
    \centering
    \includegraphics[scale=1]{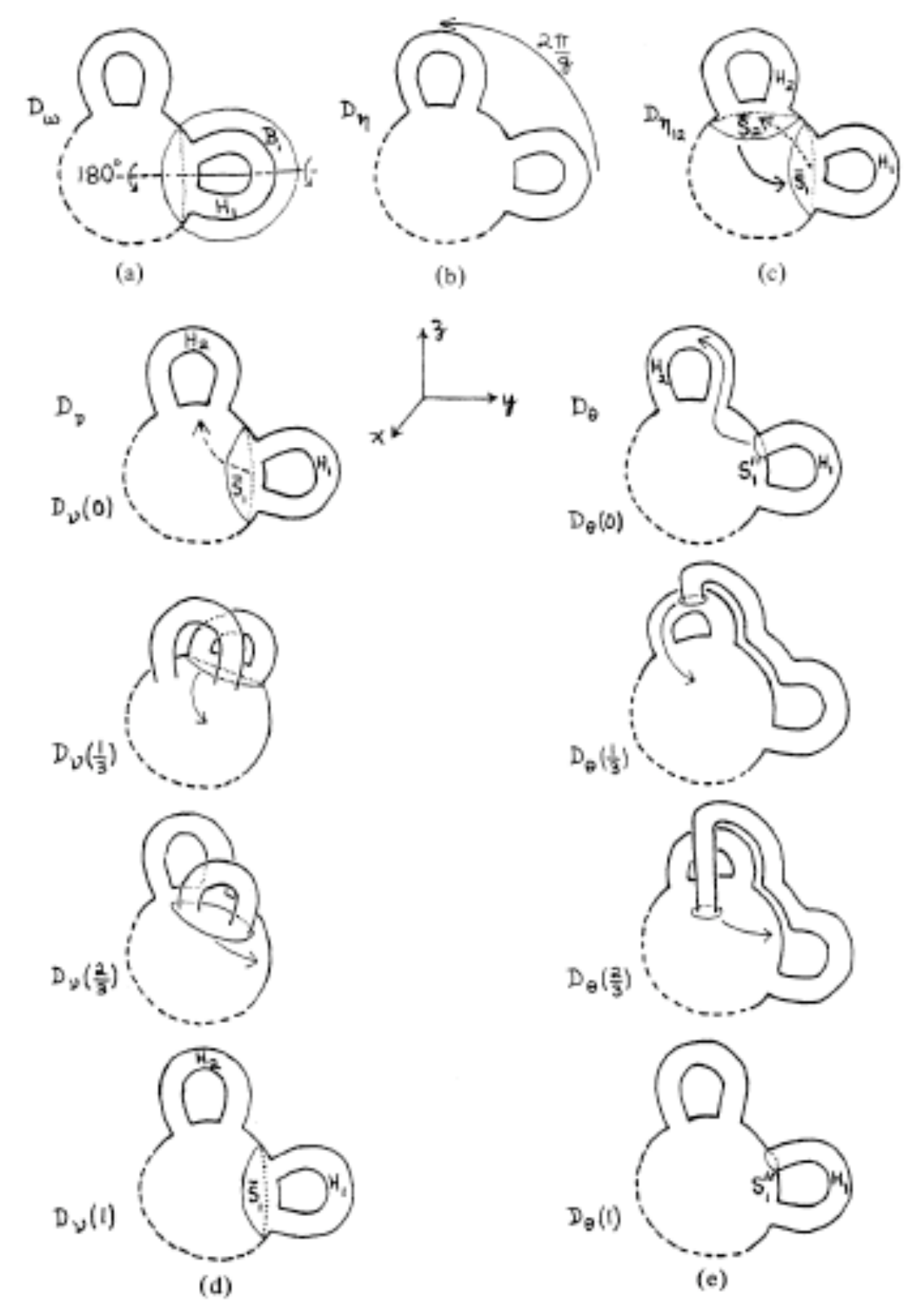}
    \caption{Powell's proposed generators}
    \label{fig:PowellPic}
    \end{figure}
    
It was noted in \cite[Lemma 1.4]{FS} that the first standard genus $1$ summand of the standardly embedded genus $g$ Heegaard surface of $S^3$ can be passed around a {\em longitude} of the second standard summand by the use of Powell's generators.  Since Powell's element $D_\nu$ similarly passes the first summand around the {\em meridian} of the second generator, it followed almost immediately that an isotopy of any standard summand around its complementary surface,  and back to its initial position, is a product of Powell's generators.  What was not noticed there, and is shown here, is that a version of the same argument, appropriately dualised, shows that Powell's $D_\nu$ is itself also a consequence of the other generators, so that Powell's actual conjecture would imply that only four generators suffice.  

The proof is mostly a sequence of pictures, modeled on those underlying the proof of \cite[Lemma 1.4]{FS}.  It will be convenient to display how to create not $D_\nu$ itself, but rather the conjugate of  $D_\nu$ by generator $D_{\eta_{12}}$, which we will denote $D'_\nu$.  This is the move which isotopes standard summand $H_2$ around a meridian of $H_1$, instead of vice versa and is illustrated in Figure \ref{fig:Dnucon}.  

\begin{figure}[ht!]
    \centering
    \includegraphics[scale=0.7]{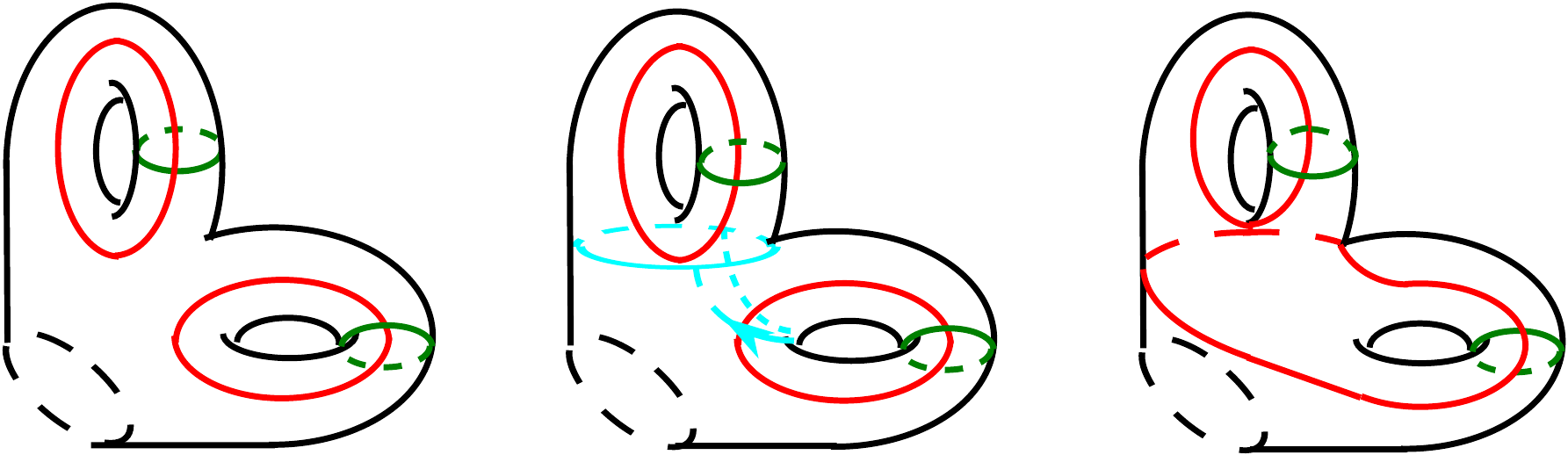}
    \caption{$D'_\nu = D_{\eta_{_{12}}} D_\nu D^{-1}_{\eta_{_{12}}}$}
    \label{fig:Dnucon}
    \end{figure}

The first image shows the meridians and longitudes of the two standard summands.  The second image shows in blue the trajectory through which the top standard summand is moved with respect to the right summand:  it is passed behind and then around the meridian of the right summand.  The effect on the meridians and the longitudes of this isotopy is shown in the third image.

Note in Figure \ref{fig:PowellPic} that the Powell generator $D_\omega$ can be described as a half-twist of the right summand shown in the images in Figure \ref{fig:Dnucon}.  The conjugate $D_{\eta_{_{12}}} D_\omega D^{-1}_{\eta_{_{12}}}$ is then the half-twist of the top summand.  Denote this move by $D'_\omega$.  Recall also from Figure \ref{fig:PowellPic} that $D_{\theta}$ denotes the move that slides the right handle over the top handle.  

\begin{lemma} \label{lemma:redund} The product $D'^{-1}_\omega D^{-1}_\theta D'_\omega D^{-1}_\theta$ has the same effect on the meridians and longitudes of the first two standard summands as $D'_\nu$ does.
\end{lemma}

\begin{proof}  This is just a matter of watching what the composition does, as shown in Figure \ref{fig:redund}.  

\begin{figure}[ht!]
    \centering
    \includegraphics[scale=0.7]{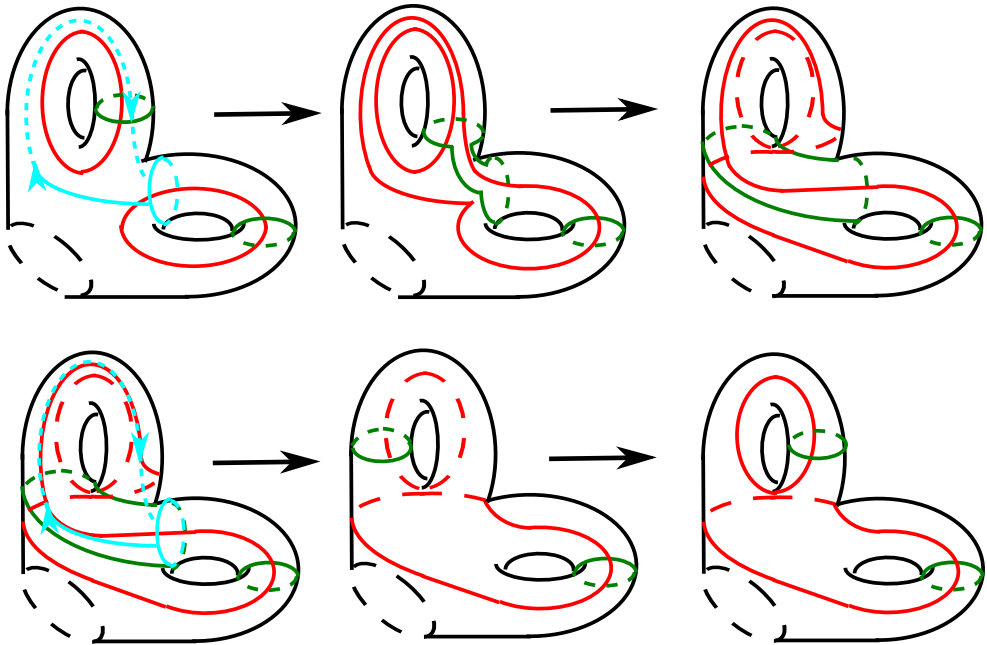}
    \caption{$D'^{-1}_\omega D^{-1}_\theta D'_\omega D^{-1}_\theta$}
    \label{fig:redund}
    \end{figure}
    
    The top row shows in blue, in the first image, the trajectory of $D^{-1}_\theta$ and then in the next image the result of this action on the meridians and longitudes.  The further effect of $D'_\omega$ is shown in the last image in the row.  The second row repeats the process for $D^{-1}_\theta$ again, and then $D'^{-1}_\omega$.
\end{proof}

\begin{thm} Powell's generator $D_\nu$ is a consequence of the three generators $D_\omega, D_{\eta_{_{12}}}, D_\theta$.   
\end{thm}

\begin{proof}  Lemma \ref{lemma:redund} illustrates that $D_\nu$ and some consequence $C$ of the three named generators have the same effect on the meridians and longitudes of the first two standard summands (and {\it ipso facto} are both the identity outside the images shown).  It follows that $D_\nu C^{-1}$ is the identity except perhaps on the the complement (in the figures) of the standard two meridians and longitudes.  This complement is a pair of pants, and any automorphism, rel $\bdd$, of a pair of pants is a product of Dehn twists on annular neighborhoods of its three boundary components.  But any Dehn twist on such a boundary component can be realized as a power of one of $D_{\eta_{_{12}}}, D_\omega$ or $D'_\omega$.
\end{proof}

\end{document}